\documentclass[11pt]{amsart}
 
\theoremstyle{plain}
\newtheorem{thm}{Theorem}[section]
\newtheorem{theorem}[thm]{Theorem}

\newtheorem{lemma}[thm]{Lemma}

\newtheorem{proposition}[thm]{Proposition}

\newtheorem{remark}[thm]{Remark}

\numberwithin{equation}{section}

\title[Properness of identity plus linear powers: part 2] {Some observations on the properness of identity plus linear powers: part 2}
\author{Tuyen Trung Truong}
\date{\today}
\address{University of Oslo, 0316 Oslo, Norway }\email{tuyentt@math.uio.no}
\begin{document}

\begin{abstract}
\noindent 

This paper develops our previous work on properness of a class of maps related to the Jacobian conjecture. The paper has two main parts: 

- In part 1, we explore properties of the set of non-proper values $S_f$ (as introduced by Z. Jelonek) of these maps. In particular, using a general criterion for non-properness of these maps, we show that under a "generic condition" (to be precise later) $S_f$ contains $0$ if it is non-empty. This result is related to a conjecture in our previous paper. We obtain this by use of a "dual set" to $S_f$, particularly designed for the special class of maps. 

- In part 2, we use the non-properness criteria obtained in our work to construct a counter-example to the proposed proof in arXiv:2002.10249 of the Jacobian conjecture.  

In the conclusion, we present some comments pertaining the Jacobian conjecture and properness of polynomial maps in general. 
 
 \end{abstract}

\maketitle


We recall that the famous Jacobian Conjecture is the following statement:

\noindent
{\bf Jacobian Conjecture.} Let $F=(F_1,\ldots ,F_m):\mathbb{C}^m\rightarrow \mathbb{C}^m$ be a polynomial map such that $JF$ (the Jacobian matrix $(\partial F_i/\partial x_j)_{1\leq i,j\leq m}$) is invertible at every point. Then, $F$ has a polynomial inverse. 

Despite efforts from many researchers and insights and techniques from many different fields, the Jacobian Conjecture is still open even in dimension $2$.  One can reduce the truth of the Jacobian conjecture to studying some properties of a special class of maps, which we present next, see \cite{druzkowski}. 

For two vectors $x,y\in \mathbb{R}^m$ we define $x*y=(x_1y_1,\ldots ,x_my_m)$, and if $x=y$ we denote $x^2=x*x$. More generally, provided a power $\alpha$ is valid for all coordinates of $x$, we denote $x^{\alpha}=(x_1^{\alpha},\ldots ,x_m^{\alpha})$. This is the case for example when $\alpha =p/q$ where $p,q$ are positive integers and $q$ is odd. It is also the case for a general rational number $\alpha =p/q$, $p$ can be negative and $q$ is odd, provided all the coordinates of $x$ are non-zero. If $B$ is a set whose all elements can take power to the $\alpha$, we then define $B^{\alpha}=\{x^{\alpha}:~x\in B\}$.  

For an $m\times m$ matrix $A$ with real coefficients, which can be thought of as a linear map $A:~\mathbb{R}^m\rightarrow \mathbb{R}^m$, we denote by $Ker(A)=\{x:~Ax=0\}$ the kernel of $A$ and $Im(A)=\{z:~z\in A(\mathbb{R}^m)\}$ the image of $A$. We denote by $A^T$ the transpose of $A$. We have the following Kernel-Image theorem in Linear Algebra: $Im(A^T)$ and $Ker(A)$ are orthogonal with respect to the inner product $<,>$, and $\mathbb{R}^m=Ker(A)\oplus Im(A^T)$. As a consequence, if $z\in Im(A^T)$ is non-zero, then $Az\not=0$. [To check this last claim, one can simply write $z=A^Tu$ and $<Az,u>=<z,A^Tu>=<z,z>$.] 

Given an $m\times m$ matrix $A$ with real coefficients, we consider the map $F_A(x)=x+(Ax)^3:~\mathbb{R}^m\rightarrow\mathbb{R}^m$. We recall that a map $f$ is proper if the pre-images of every compact sets are compact. Equivalently, a map $f$ is proper iff the image of every unbounded set is unbounded. 
We recall that $A$ is a Druzkowski matrix if $det(JF_A(x))=1$ for all $x$. It is known that the Jacobian conjecture can be reduced to  the properness of all Druzkowski maps $F_A(x)$ on $\mathbb{R}^m$, see for example the exposition in \cite{liu, truong}. It is also known, see \cite{druzkowski}, that if $f:\mathbb{R}^m\rightarrow \mathbb{R}^m$ is a polynomial map so that $f(0)=0$ and $Df(0)$ is invertible, then there  are a new variable $y\in \mathbb{R}^k$, together with two polynomial automorphisms $C,D$ with Jacobian $1$ of $\mathbb{R}^m\times \mathbb{R}^k$ so that 
\begin{eqnarray*}
C\circ (f(x),y)\circ D=F_A(x,y),
\end{eqnarray*}
where $A$ is an $(m+k)\times (m+k)$ matrix. Hence, injectivity, bijectivity and properness of a non-degenerate polynomial map from $\mathbb{R}^m$ to itself can be reduced to that of a map $F_A$. Thus, insights from maps  of the form $F_A$ give rise to insight about general polynomial maps. 

More generally, assume that $f:\mathbb{R}^m\rightarrow \mathbb{R}^{k}$ be a non-degenerate polynomial map, where $k\geq m$. We assume that $f(0)=0$ and $Df(0)$ is non-degenerate. We can assume, without loss of generality, that the left uppermost $m\times m$ minor of $Df(0)$ (constructed by the first components $f_1,\ldots ,f_m$ of $f$). We then can construct a new map 
\begin{eqnarray*}
F(x)=(f_1,\ldots, f_{m-1},f_m+f_m^2+f_{m+1}^2+\ldots +f_k^2).
\end{eqnarray*}
 It can be checked that $F(0)=0$ and $DF(0)$ is invertible. Moreover, $F$ is a polynomial map from $\mathbb{R}^m\rightarrow \mathbb{R}^m$, and it can be checked that $f$ is proper iff $F$ is proper. Therefore,  the reduction from the previous paragraph applies. 
 
 Thus, we see that the the properness of generically finite polynomial maps from affine spaces to affine spaces can be reduced to that of maps of the form $F_A$.  In the previous paper \cite{truong}, we have illustrated that studying properness of maps $F_A$ could be more advantageous than studying properness of general polynomial maps (even though, as seen above, the two questions are the same) because of the simpler form. For example, several obstruction for properness have been obtained in \cite{truong} via linear algebra. This paper aims to give more support to this viewpoint.

\subsection{The sets of non-proper values and non-proper directions} 

Given $f:\mathbb{R}^m\rightarrow \mathbb{R}^m$ a non-degenerate (= dominant, generically finite) map. Following \cite{jelonek, jelonek2}, we consider the following set:

{\bf Definition:  Non-proper value set.} We define $S_f$ to be the set of $z\in \mathbb{R}^m$ so that there is a sequence $x_n\in \mathbb{R}^m$ for which $||x_n||\rightarrow \infty$ and $f(x_n)\rightarrow z$. Since $S_f$ is contained in the image of $f$, we can call it the set of non-proper values of $f$. 

By definition, it follows that $f$ is proper iff $S_f=\emptyset$.  Many properties are known about these sets, for example they are uniruled, and for maps from $\mathbb{C}^m\rightarrow \mathbb{C}^m$ better descriptions are achieved, see  \cite{jelonek, jelonek2, jelonek-lason}. For a relevant result, we mention Corollary 3.3. in \cite{jelonek-lason}, where it is shown that if $f$ is a cubic map then $S_f$ is covered by lines and parabolas. We present first a similar, but simpler characterised, for the set $S_f$ when $f$ is of the form $F_A$. 

\begin{proposition}
Let $F_A:~\mathbb{R}^m\rightarrow \mathbb{R}^m$ be the map $x+(Ax)^3$. If $S_{F_A}$ is non-empty, then 
\begin{eqnarray*}
S_{F_A}=(S_{F_A}\cap Im(A^T))+Ker(A). 
\end{eqnarray*}
In particular, if $S_{F_A}$ is non-empty, then it is covered by lines.

\label{Proposition1}\end{proposition} 
\begin{proof}
We recall the following result in \cite{truong}: The map $F_A$ is proper iff there is a sequence $y_n\in Im(A^T)$ so that $||y_n||\rightarrow \infty$ and $Ay_n+A(Ay_n)^3$ is bounded. For one direction, whenever $x_n+y_n\mathbb{R}^m$ is a sequence with $x_n\in Ker(A)$ and $y_n\in Im(A^T)$ such that $||x_n+y_n||\rightarrow \infty$ and $F_A(x_n+y_n)$ is bounded, then $||y_n||\rightarrow \infty$ and $Ay_n+A(Ay_n)^3$ is bounded. For the other direction, if $y_n\in Im(A^T)$ so that $||y_n||\rightarrow \infty$ and $Ay_n+A(Ay_n)^3$ is bounded, then we choose $v_n\in im(A^T)$ such that $Av_n=Ay_n+A(Ay_n)^3$. Then it follows that $v_n$ is bounded, $x_n=v_n-(y_n+(Ay_n)^3)\in Ker(A)$, $||x_n+y_n||\rightarrow \infty$ and $F_A(x_n+y_n)=v_n$ is bounded. Hence, any cluster point $v$ of $v_n$, which is an element of $Im(A^T)$, is contained in $S_f$. 

Now the proof is complete by noticing that if $v\in S_f$, then $v+w\in S_f$ for all $w\in Ker(A)$. Indeed, let $x_n$ be a sequence for which $||x_n||\rightarrow \infty$ and $F_A(x_n)\rightarrow v$. Then $F_A(x_n+w)=F_A(x_n)+w\rightarrow v+w$.   
\end{proof}

Inspired by \cite{liu}, we considered in \cite{truong} a dual viewpoint about the properness of a map $F_A$, by looking on the domain of the map, by making use of special properties of $F_A$ which are not available for a general polynomial map.  To make the presentation simpler, let us denote another map $\widehat{F_A}(x)=x+A(x^3)$. It has been shown in \cite{truong} that $F_A$ is proper iff the restriction of $\widehat{F_A}$ to $Im(A)$ is proper (cf. also the proof of Proposition \ref{Proposition1}). We formalise this dual viewpoint by the following definition.

{\bf Definition: Non-proper direction set.} We let $\widehat{S}_{A}$ to be the set of $x\in Im(A)$ so that there is a sequence $x_n$ for which $||x_n||\rightarrow \infty$, $\widehat{F_A}(x_n) $ is bounded and $x_n/||x_n||\rightarrow x$. Since $\widehat{S}_{A}$ contains only points in the real unit sphere inside the domain of the  map, we can call it the non-proper direction set. 

Like $S_{F_A}$, the set  $\widehat{S}_{A}$ can also be used to characterise the properness of $F_A$. The following simple lemma provides a first connection between the two dual viewpoints. 
\begin{lemma}
Let $x_n$ be a sequence so that $||x_n||\rightarrow \infty$, $x_n+(Ax_n)^3\rightarrow \alpha$ and $x_n/||x_n||\rightarrow x_{\infty}$. Then $\alpha +\mathbb{R}x_{\infty}\subset S_{F_A}$.
\label{Lemma1}\end{lemma}
\begin{proof}
We follow the ideas in \cite{jelonek, jelonek2, jelonek-lason}. Let $x_n$ be as in the assumptions of the lemma, and put $\gamma _n=||x_n||$. Then we have $\gamma _n\rightarrow \infty$. We choose $\epsilon _n=1/\gamma _n$ and consider the sequence $x_n'=(1-\epsilon _nt)x_n$, where $t\in \mathbb{R}$ is fixed. 
It can be checked that $||x_n'||\rightarrow \infty$. Then, it can be checked that $F_A(x_n')\rightarrow \alpha -2tx_{\infty}$.

\end{proof}

In \cite{truong} we provided some obstructions for a vector $x$ to be in $\widehat{S}_{A}$. In particular, it is shown in Theorem 1.10 in the mentioned paper that if $Im(A)$ contains $Ker(A)^{1/3}$ and $x$ has all coordinates being non-zero, then $x\in \widehat{S}_{A}$ iff $x^3\in Ker(f)$ and there is $u\in  Im(A)$ for which $x+A(u*x^2)=0$. We now present the first main result of this paper, where we provide a  generalisation of Theorem 1.10 in \cite{truong} to the case where the condition $Ker(A)^{1/3}\subset Im(A)$ is not needed. After that, we will use this new result to obtain a sufficient condition for $0\in S_{F_A}$. In the proof of Theorem \ref{Theorem1}, we will use the following notation: If $x=(x_1,\ldots ,x_m)$ is a vector, then $\Delta [x]$ is the diagonal matrix whose $(i,i)$-th entry is $x_i$. 

\begin{theorem}
Let $V\subset \mathbb{R}^m$ be a vector subspace. Assume that $Im(A)\subset V$, and choose $x_{\infty}\in V$. Assume moreover that all coordinates of $x_{\infty}$  are non-zero. The following three statements are equivalent: 

1) There exists a sequence $x_n\in V$ for which $||x_n||\rightarrow \infty$, $\widehat{F_A}(x_n)$ is bounded, and $x_n/||x_n||\rightarrow x_{\infty}$. 

2) There exist sequences $z_n\in Ker(A)$ and $v_n\in V$ for which $||z_{n}||\rightarrow \infty$, $z_n^{1/3}+A(z_n^{2/3}*v_n)$ is bounded, and $z_n^{1/3}/||z_{n}^{1/3}||\rightarrow x_{\infty}$.   

3) We must have $x_{\infty}^3\in Ker(A)$, and there exists $v\in V$ so that $x_{\infty}+A(x_{\infty}^2*v)=0$. 

\label{Theorem1}\end{theorem}

\begin{proof}
{\bf Proof of 1) $\Rightarrow $ 2)}: Let $x_n$ be in the assumptions of part 1). We denote $||x_n||=\gamma _n$, then $\gamma _n\rightarrow \infty$. As in the proof of Theorem 1.10 in \cite{truong}, we  can write $x_n^3=z_n+u_n$, with $z_n\in Ker(A)$ and $u_n\in Im(A^T)$ for which $x_n+Au_n$ is bounded. Since there is a constant $c>0$ for which $||Au||\geq c||u||$ for all $u\in Im(A^T)$, it follows that $||u_n||\sim ||x_n||=\gamma _n$ and hence $||z_n||\sim ||x_n^3||\sim \gamma _n^3$. Also, it follows from the assumption on $x_n$ that $z_n^{1/3}/\gamma _n\rightarrow x_{\infty}$. Since all coordinates of $x_{\infty}$ are non-zero, it follows that all coordinates of $z_n$ grow like $\gamma _n^3$. Hence, writing $x_n=(z_n+u_n)^{1/3}=z_n^{1/3}*(1+u_n*z_n^{-1})^{1/3}$ (here the $1$ in the bracket means the vector all coordinates are $1$)  and using Taylor's expansion, we obtain
\begin{eqnarray*}
x_n=z_n^{1/3}+\frac{1}{3}u_n*z_{n}^{-2/3} + O(1/\gamma _n^3). 
\end{eqnarray*}
 Here we use the big-O notation. Hence, we obtain 
 \begin{eqnarray*}
 u_n=3z_n^{2/3}*x_n-z_n+O(1/\frac{\gamma _n}).
 \end{eqnarray*}
Hence, since $z_n\in Ker(A)$, it follows that 
\begin{eqnarray*}
O(1)&=&x_n+Au_n=(z_n^{1/3}+\frac{1}{3}u_n*z_{n}^{-2/3} + O(1/\gamma _n^3))+A(3z_n^{2/3}*x_n-z_n+O(1/\frac{\gamma _n}))\\
&=&z_n^{1/3}+A(3z_n^{2/3}*x_n).
\end{eqnarray*}
Hence, the conclusion of part 2) follows when we choose $v_n=3x_n\in V$.  

{\bf Proof of 2) $\Rightarrow $ 1)}:  Let $z_n,v_n$ be the sequences in the assumptions of part 2). Let $\gamma _n=||z_n^{1/3}||$, then $\gamma _n\rightarrow\infty$.  By replacing $z_n,v_n$ by $z_n'=\epsilon _n^3z_n$ and $v_n'=v_n/\epsilon _n$, for some appropriate sequence of positive real numbers $\epsilon _n\rightarrow 0$, if needed, we can assume furthermore that
\begin{eqnarray*}
z_n^{1/3}+A(z_n^{2/3}*v_n)=O(1/\gamma _n^3). 
\end{eqnarray*}
(Indeed, it suffices to choose $\epsilon _n=||z_n||^{-3/10}$ for all $n$.)

We will first show that we can assume further that for the above  $v_n$ we have $||z_n^{2/3}*v_n)||\sim ||z_n^{1/3}||=\gamma _n$. Indeed, the above can be written as
\begin{eqnarray*}
A.(z_n^{2/3}.v_n)=-z_n^{1/3}+O(1/\gamma _n^3). 
\end{eqnarray*}
 
We can also write the above as:
\begin{eqnarray*}
\Delta [z_n^{-2/3}].A.\Delta [z_n^{2/3}].v_n=\Delta [z_n^{-2/3}].(-z_n^{1/3}+O(1/\gamma _n)). 
\end{eqnarray*}
We define 
\begin{eqnarray*}
A_n=\Delta [z_n^{-2/3}].A.\Delta [z_n^{2/3}]=\Delta [\gamma _n^2z_n^{-2/3}].A.\Delta [z_n^{2/3}.\gamma _n^{-2}],
\end{eqnarray*}
which is a linear map $A_n:V\rightarrow \mathbb{R}^m$. The second representation of $A_n$, together with the fact that $z_n^{1/3}/\gamma _n\rightarrow x_{\infty}$ and the assumption that all coordinates of $x_{\infty}$ are non-zero, implies that $A_n\rightarrow A_{\infty}=\Delta [z_{\infty}^{-2/3}].A.\Delta [z_{\infty}^{2/3}]$. Therefore, there is a constant $C>0$ so that for all $n$ and all $w_n \in Im(A_n)$, there is $v_n\in V$ for which $A(v_n)=w_n$ and $||v_n||\leq C||w_n||$. Applying this with $w_n=\Delta [z_n{-2/3}].(-z_n^{1/3}+O(1/\gamma _n))\in Im(A_n)$ and that $||w_n||\sim \gamma _n^{-1}$ by calculating, we can find $v_n\in V$ so that $||v_n||\sim \gamma _n^{-1} $ and $z_n^{1/3}+A(z_n^{2/3}*v_n)=O(1/\gamma _n^3)$. 

Because $Im(A)\subset V$ by assumption and $v_n\in V$, the following vector is in $V$: 
\begin{eqnarray*}
x_n=-A(z_n^{2/3}*v_n)+\frac{1}{3}v_n.
\end{eqnarray*}
By the properties of $v_n$ and $z_n$, we have 
\begin{eqnarray*}
x_n=z_n^{1/3}+\frac{1}{3}v_n+O(1/\gamma _n^3). 
\end{eqnarray*}
Since $||v_n||=O(1/\gamma _n)$ and $||z_n||^{1/3}=\gamma _n$ and $\gamma _n\rightarrow\infty$, we obtain: 
\begin{eqnarray*}
x_n^3=z_n+z_n^{2/3}*v_n+O(1/\gamma _n). 
\end{eqnarray*}
Hence, from that $z_n\in Ker(A)$, $z_n^{1/3}+A(z_n^{2/3}*v_n)=O(1/\gamma _n^3)$ and $||v_n||=O(1/\gamma _n)$, we obtain finally: 
\begin{eqnarray*}
x_n+A(x_n^3)&=&[z_n^{1/3}+\frac{1}{3}v_n+O(1/\gamma _n^3)]+A[z_n+z_n^{2/3}*v_n+O(1/\gamma _n)]\\
&=&z_n^{1/3}+A(z_n^{2/3}*v_n)+O(1/\gamma _n)=O(1/\gamma _n). 
\end{eqnarray*}

{\bf Proof of 2) $\Rightarrow$ 3):}  Let $z_n\in Ker(A)$ and  $v_n\in V$ be as in the assumptions of part 2). Then, with $||z_n^{1/3}||=\gamma _n$, as in the proof of Part 2) $\Rightarrow$ Part 1) we can assume that for $\gamma _n\rightarrow\infty$, $||v_n||=O(1/\gamma _n)$ and $z_n^{1/3}/\gamma _n\rightarrow x_{\infty}$. Therefore, $x_{\infty}^3$, which is the limit of $z_n/\gamma _n^3$, must belong to $Ker(A)$. Moreover, from 
\begin{eqnarray*}
z_n^{1/3}+A(z_n^{2/3}*v_n)=O(1),
\end{eqnarray*}
when dividing by $\gamma _n$ and taking limit we obtain: $x_{\infty}+A(x_{\infty}^2*v)=0$.  Here $v$ is the limit of $\gamma _nv_n$, and hence must be in $V$. 

{\bf Proof of 3) $\Rightarrow $ 1): } This is given in the proof of Theorem 1.10 in \cite{truong}. Let $x_{\infty}$ and $v$ be as in the assumptions of part 3). Then we simply choose
\begin{eqnarray*}
x_n=\gamma _nx_{\infty}+\frac{1}{3\gamma _n}v,
\end{eqnarray*}
 where $\gamma _n\rightarrow\infty$, and check that $x_n\in V$ and $x_n+A(x_n^3)$ is bounded (indeed, converges to $0$).   
 \end{proof}

 The proof of Theorem \ref{Theorem1} also gives right away the following interesting result - the second main result in this paper - which we state separately to emphasise. 
 \begin{theorem}
 Let $V=Im(A)$ and $x_{\infty}\in V$. Assume further that all coordinates of $x_{\infty}$ are non-zero. If $x_{\infty}\in \widehat{S}_A$, then $0\in S_{F_A}$.  
 \label{Theorem2}\end{theorem}
 \begin{proof}
 Indeed, if $x_{\infty}\in \widehat{S}_A$, then by the proof of Part 2) $\Rightarrow $ Part 1) (or of Part 3) $\Rightarrow$ Part 1) ) in Theorem \ref{Theorem1}, there is a sequence $x_n\in Im(A)$ for which $||x_n||\rightarrow \infty$ and $x_n+A(x_n^3)\rightarrow 0$. Then Proposition \ref{Proposition1} concludes the proof. 
 \end{proof}

The relations of Theorems \ref{Theorem1} and \ref{Theorem2} to Conjecture C in \cite{truong} will be described in the conclusion part at the end of this paper. 

\subsection{Analysis of the proposed proof in \cite{liu}} In this subsection we analyse the proposed proof of the Jacobian conjecture given in \cite{liu}.

We will consider the following special class of matrices $A$. 

{\bf Definition.} We define $\mathcal{Z}$ to be the class of matrices $A$ so that for every $\lambda \in \mathbb{R}$, the equation $x+\lambda (Ax)^3=0$ has only one solution $x=0$. 

 It is well-known that if $A$ is a Druzkowski matrix then $A\in \mathcal{Z}$. Indeed, if we denote, for a vector $v$, by $\Delta [v]$ the diagonal matrix whose diagonal entries are coordinates of $v$, then a matrix $A$ is Druzkowski iff $A.\Delta [(Ax)^2]$ (and hence $\Delta [(Ax)^2].A$) is nilpotent for all $x$. Hence, the equation $x+\lambda (Ax)^3$, which can be deduced to $(Id+\lambda\Delta [(Ax)^2].A)x=0$, has only one solution $x=0$ since the matrix $Id+\lambda\Delta [(Ax)^2].A$ is invertible. 

In \cite{liu},  Liu gave a proof of the following claim.

{\bf Claim 1.} For every $A\in \mathcal{Z}$, the map $F_A(x)$ is proper. 

If Claim 1 were correct, then the Jacobian conjecture would follow, by Hadamard's theorem and Druzkowski's reduction, see \cite{druzkowski, liu, truong}. However, in this paper we will construct a counter-example to Claim 1. We do not know specifically where the proof in \cite{liu} breaks down, but we speculate that it is around Lemma 2.17 in \cite{liu}. 

Strictly speaking, \cite{liu} did not state Claim 1, but only the special case of it when $A$ is a Druzkowski matrix, which is again enough (in fact, equivalent to) for the Jacobian conjecture. However, if the readers scrutinise that paper, then it is found that the fact $A$ is a Druzkowski matrix is used only through Proposition 2.2, whose conclusion is exactly the definition of the class $\mathcal{Z}$. Hence, if the proof in \cite{liu} were correct, then it would give rise to a proof of Claim 1. 

The remaining of this paper is to prove the following. 
 
{\bf Claim 2.} The assertion of Claim 1 is incorrect.  
 
The proof of Claim 2 uses some sufficient conditions for a map $F_A$ to be non-proper. More precisely, we will use the  following, which is a verbatim copy of Remark 1.16 in \cite{truong}. 

\begin{remark}

We now use Corollary 1.15 to construct examples of $3\times 3$ matrices $A$ for which $F_A(x)$ is not proper. We write such a matrix as:
\[ \left( \begin{array}{ccc}
a_{1,1}&a_{1,2}&a_{1,3}\\
a_{2,1}&a_{2,2}&a_{2,3}\\
a_{3,1}&a_{3,2}&a_{3,3}\\
\end{array}\right) \]
We want $A$ to be a matrix of rank $2$, whose kernel $Ker(A)$ is generated by the vector $x_{\infty}=(1,1,1)$ and $x_{\infty}\in Im(AA^T)\cap Im(A^2A^T)$. Hence, after permuting of coordinates, from the assumption that $A$ is of rank $2$ we get 
\begin{eqnarray*}
(a_{3,1},a_{3,2},a_{3,3})=\lambda (a_{1,1},a_{1,2},a_{1,3})+\mu (a_{2,1},a_{2,2},a_{2,3}),
\end{eqnarray*}
for some real numbers $\lambda $ and $\mu$. 

The condition that $x_{\infty}\in Ker(A)$ is generated into  two equations $a_{1,1}+a_{1,2}+a_{1,3}=0$ and $a_{2,1}+a_{2,2}+a_{2,3}=0$. 

Since $Im(AA^T)=Im(A)$ (recall that $\mathbb{R}^m=Im(A^T)\oplus Ker(A)$), we have that $Im(AA^T)=\{(x,y,z):~z=\lambda x+\mu y\}$. Hence, the condition that $x_{\infty}\in Im(AA^T)$ is generated into that $1=\lambda +\mu$. 

It then follows that $Im(AA^T)=Im(A)$ is generated by $x_{\infty}$ and $(1,0,\lambda)$. Since $Im(A^2A^T)=Im(A^2)=A(Im(A))$ and $Ax_{\infty}=0$,  the condition that $x_{\infty}\in Im(A^2)$ is generated into that: there is a real number $c$ for which $A.(c,0,\lambda c)=x_{\infty}$. This is translated into that $a_{1,1}+a_{1,3}\lambda =a_{2,1}+a_{2,3}\lambda \not= 0$. 

We note that there are 4  constraints between the parameters $a_{1,1},a_{1,2},a_{1,3},a_{2,1},a_{2,2},a_{2,3},\lambda ,\mu$ (and an inequality $a_{2,1}+a_{2,3}\lambda \not= 0$), which give rise to a constructible set of dimension $4$.

\end{remark}
  
We now look at the following special case of the above remark: We choose $\lambda =1$ and $\mu =0$, and define $\alpha: =a_{1,1}+a_{1,3}=a_{2,1}+a_{2,3}\not= 0$. In this case, the matrices in the remark has the form: 

 \[ \left( \begin{array}{ccc}
a_{1,1}&a_{1,2}&a_{1,3}\\
a_{2,1}&a_{2,2}&a_{2,3}\\
a_{1,1}&a_{1,2}&a_{1,3}\\
\end{array}\right) \]

To make sure that $A$ has corank $1$, we add one more condition on determinants of $2\times 2$ minors, such as $a_{1,1}a_{2,3}-a_{1,3}a_{2,1}\not= 0$. For an explicit example, we can choose $a_{1,1}=1$, $a_{1,2}=-5$, $a_{1,3}=4$, $a_{2,1}=2$, $a_{2,2}=-5$ and $a_{2,3}=3$. In this case $\alpha =5$.  

By the above remark, for all such matrices $A$, the map $F_A(x)=x+(Ax)^3$ is non-proper. Here for the convenience of the readers, we reproduce the demonstration from \cite{truong}. We put $x_{\infty}=(1,1,1)$. The assumption  $\alpha\not= 0$ implies that for $x_{\alpha}=\frac{1}{\alpha}(1,0,1)$ we have $A.x_{\alpha}=x_{\infty}$. Note that both $x_{\infty}$ and $x_{\alpha}$ belong to $V=Im(A)$, and also that $x_{\infty}^3=x_{\infty}\in Ker(A)$, because of the conditions $a_{1,1}+a_{1,2}+a_{1,3}=0$ and $a_{2,1}+a_{2,2}+a_{2,3}=0$. We choose a sequence $\gamma _n\rightarrow\infty$ and define 
\begin{eqnarray*}
u_n=\gamma _nx_{\infty}-\frac{1}{3\gamma _n}x_{\alpha}. 
\end{eqnarray*}
Then
\begin{eqnarray*}
u_n^3&=&\gamma _n^3x_{\infty}-\gamma _nx_{\alpha}+O(1/\gamma _n),\\
u_n+A(u_n^3)&=&[\gamma _nx_{\infty}]+[A\gamma _n^3x_{\infty}-A\gamma _nx_{\alpha}]+O(1/\gamma _n)=O(1/\gamma _n).
\end{eqnarray*} 
Here $O(.)$ is the usual big-O notation. Note that $||u_n||\sim \gamma _n$, and hence since $u_n\in Im(A)$, there are $y_n\in Im(A^T)$ so that $Ay_n=u_n$ and $||y_n||\sim ||u_n||\sim \gamma _n$.  Also, since $u_n+A(u_n^3)\in Im(A)$, there is $v_n$ such that $Av_n=u_n+A(u_n^3)$ and $||v_n||\sim ||u_n+A(u_n^3)||\sim O(1/\gamma _n)$. Moreover, since $u_n+A(u_n^3)=A(y_n+(Ay_n)^3)$, we have that $x_n=v_n-[y_n+(Ay_n)^3]\in Ker(A)$. Hence, if we choose $z_n=x_n+y_n$, we have that $||z_n||\geq ||y_n||\rightarrow \infty$. Also, $z_n+(Az_n)^3=(x_n+y_n)+(Ay_n)^3=v_n=O(1/\gamma _n)$. Hence, $F_A$ is non-proper. 

[For the readers' convenience, here we make the above calculations even more explicit. We choose the matrix $A$ to be
 \[ \left( \begin{array}{ccc}
1&-5&4\\
2&-5&3\\
1&-5&4\\
\end{array}\right) \]
In this case $\alpha =5$, and $x_{\infty}=(1,1,1)$ and $x_{\alpha}=(1,0,1)/5$. 

We note that $x_{\infty}^i*x_{\alpha }^j=$ $x_{\infty}$ if $j=0$, and $=$ $(1/5)^jx_{\alpha}$ if $j>0$.

Then $u_n^3=\gamma _n^3x_{\infty}-(\gamma _n-\frac{1}{75\gamma _n}+\frac{1}{3375\gamma _n^3})x_{\alpha}$, and $u_n+A(u_n^3)=--\frac{1}{3\gamma _n}x_{\alpha } + (\frac{1}{75\gamma _n}-\frac{1}{3375\gamma _n^3})A(x_{\alpha})$. 

$Im(A^T)$ is generated by the rows of $A$, that is $(1,-5,4)$ and $(2,-5,3)$. We can compute for $a,b\in \mathbb{R}$:
\begin{eqnarray*}
A.(a(1,-5,4)+b(2,-5,3))=a(42,39,42)+b(39,38,39). 
\end{eqnarray*}
 Since the matrix
  \[ \left( \begin{array}{cc}
42&39\\
39&38\\
\end{array}\right) \]
is invertible, for every $c,d\in \mathbb{R}$, there exist unique values $a,b$ so that $(c,d,c)=A.(a(1,-5,4)+b(2,-5,3))$. We choose $x_{\infty}^*,x_{\alpha}^*\in Im(A^T)$ to be the unique solutions to $Ax_{\infty}^*=x_{\infty}$ and $Ax_{\alpha}^*=x_{\alpha}$.

Hence, we can choose
\begin{eqnarray*}
v_n&=&-\frac{1}{3\gamma _n}x_{\alpha}^*+(\frac{1}{75\gamma _n}-\frac{1}{3375\gamma _n^3})x_{\alpha},\\
y_n&=&\gamma _nx_{\infty}^*-\frac{1}{3\gamma _n}x_{\alpha}^*. 
\end{eqnarray*}
 
We then get 
\begin{eqnarray*}
x_n=v_n-[y_n+(Ay_n)^3]=-\gamma _n^3x_{\infty}-\gamma _n(x_{\infty}^*-x_{\alpha}).
\end{eqnarray*}
We note that $x_n\in Ker(A)$, since $x_{\infty}\in Ker(A)$ and $Ax_{\infty}^*=x_{\infty}=Ax_{\alpha}$. 
 ]

Now we show that such matrices are in class $\mathcal{Z}$, and hence deduce that Claim 1 is incorrect.

 Let $\lambda \in \mathbb{R}$, we will show that the equation $x+\lambda (Ax)^3=0$ has only one solution $x=0$. Indeed, let $x=(x_1,x_2,x_3)$ be a solution to $x+\lambda (Ax)^3=0$. If we write $Ax=(z_1,z_2,z_3)$ then $z_1=z_3$. Hence, the equations $x_1+\lambda z_1^3=0$ and $x_3+\lambda z_3^3=0$, as well as $z_1=z_3$, imply that $x_1=x_3$, and hence $x=(x_1,x_2,x_1)$. Then, using $$a_{1,1}+a_{1,3}=-a_{1,3}=a_{2,1}+a_{2,3}=-a_{2,3}=\alpha,$$ we have 
\begin{eqnarray*}
A.x&=&(a_{1,1}x_1+a_{1,3}x_1+a_{1,2}x_2,a_{2,1}x_1+a_{2,3}x_1+a_{2,2}x_2,a_{1,1}x_1+a_{1,3}x_1+a_{1,2}x_2)\\
&=&(\alpha (x_1-x_2),\alpha (x_1-x_2),\alpha (x_1-x_2)).
\end{eqnarray*} 
 Hence
 \begin{eqnarray*}
 \lambda (Ax)^3=\lambda \alpha ^3(x_1-x_2)^3(1,1,1). 
 \end{eqnarray*}
 
 Thus, the equation $x+\lambda (Ax)^3=0$ implies that $x_1=x_2$. However, in that case we have $Ax=0$, and hence $x=-\lambda (Ax)^3=0$, as wanted.

\subsection{Conclusions and acknowledgements}

In this paper we have illustrated further the advantage of checking properness of maps of the special form $F_A(x)$ over that of general polynomials, continuing the work from \cite{truong}.  We obtain various simple and linear algebraic criteria for both properness and non-properness of such maps. 

In \cite{truong}, we proposed the so-called Conjecture C which says that the set of matrices $A$ for which $F_A$ is proper is a constructible set. The stronger version of Conjecture C even mentions a specific system of algebraic equations and inequalities given in Theorem 1.11 in \cite{truong}. Note that if this stronger version of Conjecture C as mentioned holds, then the arguments in \cite{truong} shows that the following conjecture also holds: 

{\bf Conjecture D.} If a map $F_A(x)$ is non-proper, then $0\in S_{F_A}$. 

Theorem \ref{Theorem1} confirms Conjecture C in the "generic case", and likewise Theorem \ref{Theorem2} confirms Conjecture D in the "generic case". Here, it is not the same as "genericity" as commonly used in Algebraic Geometry. Rather, it is heuristic. Theorems \ref{Theorem1} and \ref{Theorem2} concern a vector whose all coordinates are non-zero. Since the set of vectors whose all coordinates are non-zero is dense in the set of all vectors, Theorem \ref{Theorem2} is in some sense a generic statement of Conjecture D. 

While we have shown that the proposed proof of the Jacobian conjecture in \cite{liu} - showing properness of maps in the class $\mathcal{Z}$-  is invalid,  we notice that the ideas put forward are fruitful in some perspectives. For example, the idea of looking at the non-proper direction set $\widehat{S}_A$ (dual to the non-proper value set $S_{f}$ introduced in \cite{jelonek, jelonek2}) and the useful observation that if $w\in Ker(A)$ then $F_A(x+w)=F_A(x)+w$ have been initiated in \cite{liu}. In fact, this paper and \cite{truong} get inspiration from \cite{liu}. 

Overall, we argue that the approach of showing properness of Druzkowski maps can be still promising, given the ease at it one can provide criteria for properness/non-properness of these maps.  Also, studying properness of general polynomials is interesting per se - with many applications and connections to topology, geometry, dynamical systems and other fields -   and can provide useful obstructions which one can attempt first in trying to solve the Jacobian conjecture. Indeed, since the class of Druzkowski matrices is complicated in higher dimensions, there is no hope to try to explore the Jacobian conjecture by classifying Druzkowski matrices. In contrast, a common approach is to select some easy to describe consequences of being Druzkowski matrices, and try to solve the Jacobian conjecture from these special properties only. More specifically, one could choose a constructible set $\Sigma$, knowing to contain all Druzkowski matrices, which is easy to describe, and try to prove  $F_A$ is either injective or proper, for $A\subset \Sigma$. When doing so, one needs to caution so that $\Sigma$ avoids the obstructions such as those in \cite{truong} and in this paper. We concern more closely Theorem \ref{Theorem1} in this paper and Theorems 1.10 and 1.11 in \cite{truong}. Given the "generic" nature of the assumptions in these results (as described in the discussion after the statement of Conjecture D), our guess is that if $\Sigma$ is not the correct class of matrices, it should fail for the criteria in the mentioned theorems. For example, for the class $\Sigma =\mathcal{Z}$ considered in \cite{liu}, the obstruction in Theorem 1.10 is not avoidable, and hence the approach therein is invalid.   

On the other hand, if one does not believe in the Jacobian conjecture, one can turn the faces of the coin around and use the above obstruction to guide finding a counter-example. The obstructions can also be used to construct or detect interesting proper or non-proper polynomial maps. One such example which we have plan to explore in the future is to test the Druzkowski form of Pinchuk's polynomial. More generally, after proving results on properness/non-properness of maps of the form $F_A(x)$, one can then trace back the construction in \cite{druzkowski} to apply back to original polynomials of more general form. For example, one idea for new proofs of results in \cite{jelonek, jelonek2, jelonek-lason} could be to first obtain the results for the associated maps of the form $F_A(x)$ and then translate backward.   

Let us close this by mentioning that there are other approaches of solving Jacobian's conjecture. For example, checking that Druzkowski maps are injective is also a popular. For Druzkowski maps, the injectivity is reduced to the non-existence of non-zero solutions to a certain polynomial equations of degree $3$. In \cite{truong2}, we verified by computer experiments that this gives rise to a very quick test. Therefore, we proposed therein to turn this question to a question on non-existence of any solution, and hence Hilbert's Nullstellensatz can be used. Roughly speaking, we want to show that a  certain system $\{f_1=\ldots =f_k=0\}$ has no solution on $\mathbb{C}$. This is equivalent  to that there are polynomials $g_1,\ldots ,g_k$ so that $\sum f_jg_j=1$. Then we can use computers to find these $g_1,\ldots ,g_k$ quickly, and then can hope to get some intelligent patterns from observing the experimental results.

{\bf Acknowlegements.} We thank Jiang Liu for helpful correspondences. This work is supported by Young Research Talents grant 300814 from Research Council of Norway.


\begin{thebibliography}{XXXXX}

\bibitem{druzkowski} L. Druzkowski, An effective approach to Keller's Jacobian conjecture, Math. Ann. 264 (1983), 303--313. 

\bibitem{jelonek-lason} Z. Jelonek and M. Laso\'n, Quantitative properties of the non-properness set of a polynomial map, arXiv: 1411.5011.

\bibitem{jelonek2} Z. Jelonek, Testing sets for properness of polynomial mappings, Math. Ann. 315 (1999), 1--35.

\bibitem{jelonek} Z. Jelonek, The set of points at which the polynomial mapping is not proper, Ann. Polon. Math. 58 (1993), 259--266.

\bibitem{liu} J. Liu, An optimization approach to Jacobian Conjecture, arXiv:2002.10249.


\bibitem{truong} T. T. Truong,  Some observations on the properness of identity plus linear powers, arXiv:2004.03309. 

\bibitem{truong2} T. T. Truong,  Some new theoretical and computational results around the Jacobian conjecture, revised and extended from arXiv:1503.08733. Accepted in International Journal of Mathematics. 


\end{thebibliography}
\end{document}